\documentclass[11pt,a4paper,onecolumn,notitlepage]{article}

\usepackage{cmap} 
\usepackage{refcount}
\usepackage[pdfpagemode=None,
pagebackref=true,
bookmarksopen=false,
hyperindex=true,
colorlinks=true,
linkcolor=blue,
citecolor=blue,
pdftitle={Noncommutative Orlicz spaces over W*-algebras},
pdfauthor={Ryszard Pawel Kostecki}]{hyperref}
\usepackage[utf8x]{inputenc}
\usepackage[LGR,OT2,T1]{fontenc}
\usepackage[polutonikogreek,english]{babel}
\usepackage{CJK}
\usepackage{cyrillic}
\usepackage[all]{xy}
\usepackage{amsmath}
\usepackage{amssymb}
\usepackage{amsthm}
\usepackage{mathrsfs} 
\usepackage{upgreek} 
\usepackage{sectsty}
\usepackage{bibspacing}
\usepackage{rpk}
\usepackage[pdftex,final]{graphicx}
\renewcommand*{\backref}[1]{}
\renewcommand*{\backrefalt}[4]{%
  \ifcase #1 %
    \relax
  \or
    $\uparrow$~#2.
  \else
    $\uparrow$~#2.
  \fi%
}

\newcommand{\cyrrm}[1]{\mbox{\fontencoding{OT2}\fontfamily{wncyr}\selectfont#1}}

\makeindex
\begin{document}
\thispagestyle{empty}
\selectlanguage{english}
\setcounter{tocdepth}{4}
\setcounter{secnumdepth}{4}
\renewcommand*{\thefootnote}{\fnsymbol{footnote}}
\begin{center}
{\huge\textbf{Noncommutative Orlicz spaces over W$^*$-algebras}}{\tiny\\\ \\\ \\}
{\LARGE Ryszard Pawe{\l} Kostecki{\tiny\\\ \\}{\small \textit{Perimeter Institute of Theoretical Physics, 31 Caroline Street North, Waterloo, Ontario N2L 2Y5, Canada}}\\{\small\vspace{3mm}\texttt{ryszard.kostecki@fuw.edu.pl}}}{\small\\\ \\}
{\today}
\end{center}
\begin{abstract}
\noindent Using the Falcone--Takesaki theory of noncommutative integration and Kosaki's canonical representation, we construct a family of noncommutative Orlicz spaces that are associated to an arbitrary W$^*$-algebra $\N$ without any choice of weight involved, and we show that this construction is functorial over the category of W$^*$-algebras with $*$-isomorphisms as arrows. Under a choice of representation, these spaces are isometrically isomorphic to Kunze's noncommutative Orlicz spaces over crossed products $\N\rtimes_{\sigma^\psi}\RR$.\\

\noindent{\scriptsize{MSC2010: 46L52, 46L51, 46M15.}\\
\noindent\scriptsize{Keywords: noncommutative Orlicz spaces, W$^*$-algebras.}}
\end{abstract}
\setcounter{footnote}{0}
\renewcommand*{\thefootnote}{\arabic{footnote}}
\section{Introduction}
The construction of a standard representation by Connes, Araki and Haagerup and its further canonical refinement by Kosaki allows one to assign a canonical $L_2(\N)$ space to every W$^*$-algebra $\N$, without any choice of a weight on $\N$ involved. This leads to a question: is it possible to develop the theory of noncommutative integration and $L_p(\N)$ spaces for arbitrary W$^*$-algebras $\N$, analogously to consideration of the Hilbert--Schmidt space as a member $\mathfrak{G}_2(\H)$ of the family of von Neumann--Schatten spaces $\mathfrak{G}_p(\H)$, and consideration of a Hilbert space $L_2(\X,\mho(\X),\tmu)$ as a member of a family of Riesz--Radon--Dunford spaces $L_p(\X,\mho(\X),\tmu)$? Falcone and Takesaki \cite{Falcone:Takesaki:2001} answered this question in the affirmative by an explicit construction, showing that to any W$^*$-algebra $\N$ there is a corresponding noncommutative $L_p(\N)$ space, and furthermore, this assignment is functorial.

In this paper we extend their result by constructing a  family of noncommutative Orlicz spaces $L_\Orlicz(\core)$ that is associated to every W$^*$-algebra $\N$ for any choice of an Orlicz function $\Orlicz$. The key aspect of this construction is its functoriality: every $*$-isomorphism of W$^*$-algebras induces a corresponding isometric isomorphism of noncommutative Orlicz spaces. We also infer basic properties of these spaces. These spaces, under a choice of any faithful normal semi-finite weight $\psi$ on $\N$, are isometrically isomorphic to Kunze's \cite{Kunze:1990} noncommutative Orlicz spaces $L_\Orlicz(\N\rtimes_{\sigma^\psi}\RR)$, which is induced by a weight-dependent unitary isomorphism $\core\iso\N\rtimes_{\sigma^\psi}\RR$. Hence, $L_\Orlicz(\core)$ are not \textit{the} canonical noncommutative Orlicz spaces over $\N$, similarly as $L_p(\X\times\RR,\mho(\X\times\RR),\tmu)$ is not the canonical $L_p$ space over a corresponding boolean algebra. Yet, we hope that our method of approaching the problem, based on the functorial properties of the Falcone--Takesaki noncommutative flow of weights, can be further extended using the ideas in \cite{Labuschagne:2013}.


Section \ref{ft.section} provides an overview of the Falcone--Takesaki integration theory over arbitrary W$^*$-algebras. A brief discussion of noncommutative Orlicz spaces is contained in Section \ref{orlicz.section}. In Section \ref{functoriality.section} we introduce our construction and show its functoriality. 

For a W$^*$ algebra $\N$, we denote: the set of semi-finite faithful nomal weights on $\N$ as $\W_0(\N)$; Connes' cocycle of $\psi,\phi\in\W_0(\N)$ as $\Connes{\psi}{\phi}{t}$; Connes' spatial quotient as $\connes{\psi}{\phi}$; the Tomita--Takesaki modular automorphism of $\psi\in\W_0(\N)$ as $\sigma^\psi$. For an extended discussion of all notions related to W$^*$-algebras and noncommutative integration we refer to \cite{Kostecki:2013} as a recent overview (close to the spirit of \cite{Sakai:1971}) which has precisely the same notation and terminology as we use here, as well as to \cite{Stratila:1981,Takesaki:2003} as standard references.
\section{Falcone--Takesaki integration on arbitrary W$^*$-algebras\label{ft.section}}
The Falcone--Takesaki approach to noncommutative integration relies on the construction and properties of the core algebra $\core$ and on Masuda's \cite{Masuda:1984} reformulation of Connes' noncommutative Radon--Nikod\'{y}m type theorem. We will first briefly recall the key notions from the relative modular theory, and then move to the construction of the core algebra and integration over it.

Consider a W$^*$-algebra $\N$ and a relation $\sim_t$ on $\N\times\W_0(\N)$ defined by \cite{Falcone:Takesaki:2001}
\begin{equation}
        (x,\psi)\sim_t(y,\phi)\iff y=x\Connes{\psi}{\phi}{t}\;\;\forall x,y\in\N\;\forall\psi,\phi\in\W_0(\N).
\label{core.equivalence}
\end{equation}
The cocycle property of $\Connes{\psi}{\phi}{t}$ implies that $\sim_t$ is an equivalence relation in $\N\times\W_0(\N)$. The equivalence class $(\N\times\W_0(\N))/\sim_t$ is denoted by $\N(t)$, and its elements are denoted by $x\psi^{\ii t}$. The operations
\begin{align}
        x\psi^{\ii t}+y\psi^{\ii t}&:=(x+y)\psi^{\ii t},\\
        \lambda(x\psi^{\ii t})&:=(\lambda x)\psi^{\ii t}\;\forall\lambda\in\CC,\\
        \n{x\psi^{\ii t}}&:=\n{x},
\end{align}
equip $\N(t)$ with the structure of the Banach space, which is isometrically isomorphic to $\N$, considered as a Banach space. By definition, $\N(0)$ a W$^*$-algebra that is trivially $*$-isomorphic to $\N$. However, for $t\neq0$ the spaces $\N(t)$ are not W$^*$-algebras. 

The operations
\begin{align}
        \cdot\,:\N(t_1)\times\N(t_2)\ni(x\psi^{\ii t_1},y\psi^{\ii t_2})&\mapsto x\sigma_{t_1}^{\psi}(y)\psi^{\ii(t_1+t_2)}\in\N(t_1+t_2),\\
        ^*\,:\N(t)\ni x\psi^{\ii t}&\mapsto\sigma_{-t}^\psi(x^*)\psi^{-\ii t}\in\N(-t),
\end{align}
equip the disjoint sum $\fell(\N):=\coprod_{t\in\RR}\N(t)$ over $\N\times\RR$ with the structure of $*$-algebra. The bijections 
\begin{equation}
        \N(t)\ni x\psi^{\ii t}\mapsto(x,t)\in\N\times\RR
\label{core.pairs.bijection}
\end{equation}
allow to endow $\fell(\N)$ with the topology induced by \eqref{core.pairs.bijection} from the product topology on $\N\times\RR$ of the weak-$\star$ topology on $\N$ and the usual topology on $\RR$. This provides the Fell's Banach $^*$-algebra bundle structure on $\fell(\N)$. One can consider the Fell bundle $\fell(\N)$ as a natural algebraic structure which enables to translate between elements of $\N(t)$ at different $t\in\RR$. In order to recover an element of $\fell(\N)$ at a given $t\in\RR$, one has to select a section $\widetilde{x}:\RR\ra\fell(\N)$ of $\fell(\N)$:
\begin{equation}
        t\mapsto x(t)\psi^{\ii t}=:\widetilde{x}(t).
\end{equation}
Consider the set $\Gamma^1(\fell(\N))$ of such cross-sections of $\fell(\N)$ that are \df{integrable} in the following sense:
\begin{enumerate}
\item[i)] for any $\epsilon>0$ and any bounded interval $I\subseteq\RR$ there exists a compact subset $Y\subseteq I$ such that $\ab{I-Y}<\epsilon$ and the restriction $Y\ni t\mapsto x(t)\in\fell(\N)$ is continuous relative to the topology induced in $\fell(\N)$ by \eqref{core.pairs.bijection},
\item[ii)] $\int_\RR\dd r\,\n{x(r)}<\infty$.
\end{enumerate}
The set $\Gamma^1(\fell(\N))$ can be endowed with a multiplication, involution, and norm,
\begin{align}
        (\widetilde{x}\widetilde{y})(t)
        &:=\int_\RR\dd r\, x(r)y(t-r)
        =\left(\int_\RR\dd r\, x(r)\sigma_r^\psi(y(t-r))\right)\psi^{\ii t},\\
  \widetilde{x}^*(t)
  &:=\widetilde{x}(-t)^*
  =\sigma_t^\psi(x(-t)^*)\psi^{\ii t},\\
  \n{\widetilde{x}}
  &:=\int_\RR\dd r\,\n{x(r)},
\end{align}
thus forming a Banach $*$-algebra, denoted by $\bundlealg(\N)$. 

Falcone and Takesaki \cite{Falcone:2000,Falcone:Takesaki:2001} constructed also a suitably defined `bundle of Hilbert spaces' over $\RR$. Let $\N=\pi(\C)$ be a von Neumann algebra representing a W$^*$-algebra $\C$ in terms of a standard representation $(\H,\pi,J,\stdcone)$. The space $\H$ can be considered as a $\N$-$(\N^\comm)^o$ bimodule\footnote{That is, a Hilbert space $\H$ equipped with a normal representation $\pi_1:\N\ra\BH$ and a normal representation $\pi_2:\N^o\ra\BH$ such that $\pi_1(\N)$ and $\pi_2(\N^o)$ commute.}, with the left action of $\N$ on $\H$ given by ordinary multiplication from the left, and with the right action of $(\N^\comm)^o$ on $\H$ defined by 
\begin{equation}
        \xi x^o:=x\xi\;\;\;\forall\xi\in\H\;\;\forall x^o\in(\N^\comm)^o.
\end{equation}
Thus, the left action of $\N$ on $\H$ is just an action of a standard representation of the underlying W$^*$-algebra $\C$, while the right action of $(\N^\comm)^o$ is provided by the corresponding antirepresentation of $\C$ (that is, by commutant of a standard representation of $\C^o$). Given arbitrary $r_1,r_2\in\RR$, $\zeta_1,\zeta_2\in\H$, $\phi_1,\phi_2\in\W_0(\N)$, and $\varphi_1,\varphi_2:\W_0((\N^\comm)^o)$ the condition
\begin{equation}      
\left(\connes{\phi_1}{\varphi_1}\right)^{\ii r_1}\zeta_1=\left(\connes{\phi_2}{\varphi_2}\right)^{\ii r_2}\zeta_2\Connes{\varphi_2}{\varphi_1}{t},
\end{equation}
defines an equivalence relation
\begin{equation}
        (r_1,\phi_1,\zeta_1,\varphi_1)\;\;\sim_t\;\;(r_2,\phi_2,\zeta_2,\varphi_2)
\label{hilb.equiv}
\end{equation}
on the set $\RR\times\W_0(\N)\times\H\times\W_0((\N^\comm)^o)$. The equivalence class of the relation \eqref{hilb.equiv} is denoted by $\H(t)$, and its elements have the form
\begin{equation}
        \phi^{\ii t}\xi=\left(\connes{\phi}{\varphi}\right)^{\ii t}\xi\varphi^{\ii t},
\end{equation}
which is equivalent to
\begin{equation}
        \phi^{\ii t}\xi\varphi^{-\ii t}=\left(\connes{\phi}{\varphi}\right)^{\ii t}\xi.
\label{connes.cocycle.as.generator}
\end{equation}
Falcone and Takesaki show that $\H(t)$ is a Hilbert space independent of the choice of weights $\phi_1,\phi_2,\varphi_1,\varphi_2$ and of the choice of $r\in\RR$. This enables to form the Hilbert space bundle over $\RR$,
$\coprod_{t\in\RR}\H(t)$, and to form the Hilbert space of square-integrable cross-sections of this bundle,
\begin{equation}
        \widetilde{\H}:=\Gamma^2\left(\coprod\limits_{t\in\RR}\H(t)\right).
\end{equation}
The Hilbert space bundle $\coprod_{t\in\RR}\H(t)$ is homeomorphic to $\H\times\RR$ for \textit{any} choice of weight $\psi\in\W_0(\N)$. The left action of $\bundlealg(\N)$ on $\widetilde{\H}$ generates a von Neumann algebra $\core$ called \df{standard core} \cite{Falcone:Takesaki:2001}. For type III$_1$ factors $\N$ the standard core $\core$ is a type II$_\infty$ factor, but in general case $\core$ is not a factor. The structure of $\core$ is independent of the choice of weight on $\N$. However, for any choice of $\psi\in\W_0(\N)$ there exists a unitary map
\begin{equation}
\coreiso_\psi:\widetilde{\H}=\Gamma^2\left(\coprod_{t\in\RR}\H(t)\right)\ra L_2(\RR,\dd t;\H)\iso\H\otimes L_2(\RR,\dd t),
\end{equation}
such that
\begin{equation}
        \coreiso_\psi(\xi)(t)=\psi^{-\ii t}\xi(t)\in\H\;\;\;\forall\xi\in\widetilde{\H}.
\end{equation}
It satisfies
\begin{align}
        (\coreiso_\psi x \coreiso_\psi^*)(\xi)(t)
        &=\sigma^\psi_{-t}(x)\xi(t),\\
        (\coreiso_\psi\psi^{\ii s}\coreiso_\psi^*)(\xi)(t)
        &=\xi(t-s),\\
        (\coreiso_\psi\phi^{-\ii s}\coreiso_\psi^*)(\xi)(t)
        &=\left(\connes{\psi}{\phi^\comm}\right)^{\ii s}.
\end{align}
for all $\xi\in L_2(\RR,\dd t;\H)$, $x\in\N$, $\phi,\psi\in\W_0(\N)$, $s,t\in\RR$. This map provides a $*$-isomorphism between the standard core $\core$ on $\widetilde{\H}$ and the crossed product $\N\rtimes_{\sigma^\psi}\RR$ on $\H\otimes L_2(\RR,\dd t)$,\footnote{See \cite{Kostecki:2013} for the notions of, and further references on, crossed product, operator valued weight, dual weight, and $\N^\ext$.}
\begin{equation}
       \coreiso_\psi^*\core\coreiso_\psi=\N\rtimes_{\sigma^\psi}\RR.
\label{core.iso.map}
\end{equation}
Using the uniqueness of the standard representation up to unitary equivalence, Falcone and Takesaki \cite{Falcone:Takesaki:2001} proved that the map $\N\mapsto\core$ extends to a functor $\VNCore$ from the category $\VNIso$ of von Neumann algebras with $*$-isomorphisms to its own subcategory $\VNsfIso$ of semi-finite von Neumann algebras with $*$-isomorphisms. 

The one-parameter automorphism group of $\fell(\N)$,
\begin{equation}
        \tilde{\sigma}_s(x\phi^{\ii t}):=\ee^{-\ii ts}x\phi^{\ii t}\;\;\;\;\forall x\phi^{\ii t}\in\N(t),
\label{tilde.sigma.def}
\end{equation}
corresponding to the unitary group $\tilde{u}(s)$ on $\widetilde{\H}$ given by
\begin{equation}
        (\tilde{u}(s)\xi)(t)=\ee^{-\ii st}\xi(t)\;\;\forall t,s\in\RR\;\forall\xi\in\widetilde{\H},
\end{equation}
extends uniquely to a group of automorphisms $\tilde{\sigma}_s:\core\ra\core$. The automorphism $\tilde{\sigma}_t$ provides a weight-independent replacement for a dual automorphism $\hat{\sigma}_t^\psi$ used in Haagerup's theory \cite{Haagerup:1979:ncLp,Terp:1981}. The triple $(\tilde{\N},\RR,\tilde{\sigma})$ is a W$^*$-dynamical system
. There exist canonical isomorphisms
\begin{align}
        \core\rtimes_{\tilde{\sigma}}\RR&\iso\N\otimes\BBB(L_2(\RR,\dd\lambda)),\\
        \core_{\tilde{\sigma}}&\iso\N.
\end{align}
The action of $\tilde{\sigma}_s$ on $\core$ is integrable over $s\in\RR$, and 
\begin{equation}
        T_{\tilde{\sigma}}:\core^+\ni x\mapsto T_{\tilde{\sigma}}(x):=\int_\RR\dd s\,\tilde{\sigma}_s(x)\in\N^\ext,
\end{equation}
is an operator valued weight from $\core$ to $\core_{\tilde{\sigma}}\iso\N$. For any $\phi\in\W(\N)$, its dual weight over $\core$ is given by
\begin{equation}
        \hat{\phi}:=\tilde{\phi}\circ T_{\tilde{\sigma}}\in\W(\core).
\end{equation}
Every $\phi\in\W_0(\N)$ can be considered as an analytic generator of the one parameter group of unitaries $\{\phi^{\ii t}\mid t\in\RR\}\subseteq\core$ acting on $\widetilde{\H}$ from the right, given by
\begin{equation}
        \phi=\exp\left(-\ii\frac{\dd}{\dd t}\left(\phi^{\ii t}\right)|_{t=0}\right).
\label{phi.in.FT.as.generator}
\end{equation}
This allows to equip $\core$ with a faithful normal semi-finite trace $\taucore_\phi:\core^+\ra[0,\infty]$, 
\begin{align}
        \taucore_\phi(x):&=\lim_{\epsilon\ra^+0}\hat{\phi}((\phi^{-1}(1+\epsilon\phi^{-1})^{-1})^{1/2}x((\phi^{-1}(1+\epsilon\phi^{-1})^{-1})^{1/2})\nonumber\\
        &=\lim_{\epsilon\ra^+0}\hat{\phi}(\phi^{-1/2}(1+\epsilon\phi^{-1})^{-1/2}x\phi^{-1/2}(1+\epsilon\phi^{-1})^{-1/2})\nonumber\\
        &=\lim_{\epsilon\ra^+0}\hat{\phi}((\phi+\epsilon)^{-1/2}x(\phi+\epsilon)^{-1/2}).
\label{canonical.trace}
\end{align}
This definition is independent of the choice of weight (e.g., $\taucore_\varphi=\taucore_\psi\;\forall\varphi,\psi\in\W_0(\N)$), which follows from the fact that 
\begin{equation}
\Connes{\taucore_\phi}{\taucore_\varphi}{t}=\Connes{\taucore_\phi}{\tilde{\varphi}}{t}
\Connes{\tilde{\varphi}}{\tilde{\psi}}{t}
\Connes{\tilde{\psi}}{\taucore_\varphi}{t}=\varphi^{-\ii t}\Connes{\varphi}{\psi}{t}\psi^{\ii t}=\varphi^{-\ii t}\varphi^{\ii t}\psi^{-\ii t}\psi^{\ii t}
=1
\end{equation}
for all $\phi,\varphi\in\W_0(\N)$ and for all $t\in\RR$. This allows to write $\taucore$ instead of $\taucore_\varphi$. Moreover, $\taucore$ has the scaling property
\begin{equation}
        \taucore\circ\tilde{\sigma}_s=\ee^{-s}\taucore\;\;\;\;\forall s\in\RR.
\label{scaling}
\end{equation}
This allows to call $\taucore$ a \df{canonical trace} of $\core$. It will play the role analogous to a Haagerup's trace $\tilde{\tau}_\psi$ on a crossed product algebra $\precore=\N\rtimes_{\sigma^\psi}\RR$ \cite{Haagerup:1979:ncLp,Terp:1981}. Nevertheless, the definition \eqref{canonical.trace} is not a straightforward generalisation of Haagerup's trace. It is another type of `perturbed' construction of a weight, which is designed in this case for the purpose of direct elimination of the dependence of $\taucore_\phi$ on $\phi$.
\section{Noncommutative Orlicz spaces\label{orlicz.section}}
In this section we will first briefly recall the key notions from the theory of commutative Orlicz spaces \cite{Orlicz:1932,Orlicz:1936}, and then we will move to discussion of the noncommutative version of this theory. For a detailed treatment of the theory of commutative Orlicz spaces, as well as the associated theory of modular spaces, see \cite{Nakano:1950,Nakano:1951:book,Nakano:1951,Zaanen:1953,Krasnoselskii:Rutickii:1958,Lindenstrauss:Tzafriri:1977:1979,Musielak:1983,Maligranda:1989,Rao:Ren:1991,Chen:1996,Rao:Ren:2002}.

If $X$ is a vector space over $\KK\in\{\RR,\CC\}$, and $\Orlicz:X\ra[0,\infty]$ is a convex function satisfying
\begin{enumerate}
\item[1)] $\Orlicz(0)=0$,
\item[2)] $\Orlicz(\lambda x)=0$ $\forall\lambda>0$ $\limp$ $x=0$,
\item[3)] $\ab{\lambda}=1$ $\limp$ $\Orlicz(\lambda x)=\Orlicz(x)$,
\end{enumerate}
then $\Orlicz$ is called a \df{pseudomodular function} \cite{Musielak:Orlicz:1959}. If 2) is replaced by
\begin{enumerate}
\item[2')] $\Orlicz(x)=0$ $\limp$ $x=0$,
\end{enumerate}
then $\Orlicz$ is called a \df{modular function} \cite{Nakano:1950,Nakano:1951:book,Nakano:1951}. If 2) is replaced by 
\begin{enumerate}
\item[2'')] $x\neq0$ $\limp$ $\lim_{\lambda\ra+\infty}\Orlicz(\lambda x)=+\infty$,
\end{enumerate}
then $\Orlicz$ is called a \df{Young function} \cite{Young:1912,Birnbaum:Orlicz:1930,Birnbaum:Orlicz:1931}. A Young function $f$ on $\RR$ is said to satisfy: \df{local $\triangle_2$ condition} if{}f \cite{Birnbaum:Orlicz:1930,Birnbaum:Orlicz:1931}
\begin{equation}
	\exists\lambda>0\;\;
	\exists x_0\geq0\;\;
	\forall x\geq x_0\;\;
	f(2x)\leq\lambda f(x);
\label{local.delta.two.condition}
\end{equation}
 \df{global $\triangle_2$ condition} if{}f $x_0$ in \eqref{local.delta.two.condition} is set to $0$. A convex function $f:\RR\ra\RR^+$ is called \df{N-function} \cite{Birnbaum:Orlicz:1931} if{}f $\lim_{x\ra^+0}\frac{f(x)}{x}=0$ and $\lim_{x\ra+\infty}\frac{f(x)}{x}=+\infty$. Every Young function $f$ allows do define a \df{Young--Birnbaum--Orlicz dual} \cite{Birnbaum:Orlicz:1931,Mandelbrojt:1939}
\begin{equation}
	f^\Young:\RR\ni y\mapsto f^\Young(y):=\sup_{x\geq0}\{x\ab{y}-f(x)\}\in[0,\infty].
\end{equation}
Every YBO dual is a nondecreasing Young function, and each pair $(f,f^\Young)$ satisfies \df{Young's inequality} \cite{Young:1912}
\begin{equation}
	xy \leq f(x)+f^\Young(y)\;\;\forall x,y\in\RR.
\end{equation}
If $f$ is also an N-function, then $f^\Young{}^\Young=f$. Every modular function $\Orlicz:X\ra[0,\infty]$ determines a \df{modular space} \cite{Nakano:1950,Nakano:1951:book}
\begin{equation}
	X_\Orlicz:=\{x\in X\mid\lim_{\lambda\ra^+0}\Orlicz(\lambda x)=0\}
\end{equation}
and the \df{Morse--Transue--Nakano--Luxemburg norm} on $X_\Orlicz$ \cite{Morse:Transue:1950,Nakano:1951:book,Luxemburg:1955,Weiss:1956},
\begin{equation}
	 \n{\cdot}_\Orlicz:X_\Orlicz\ni x\mapsto\n{x}_\Orlicz:=\inf\{\lambda>0\mid\Orlicz(\lambda^{-1}x)\leq 1\}\in\RR^+,
\end{equation}
which allows to define a Banach space
\begin{equation}
	L_\Orlicz(X):=\overline{X_\Orlicz}^{\n{\cdot}_\Orlicz}.
\end{equation}

The noncommutative Orlicz spaces associated with the algebra $\BH$ of bounded operators on a Hilbert space $\H$ were implicitly introduced by Schatten \cite{Schatten:1960} as ideals in $\BH$ generated by the so-called symmetric gauge functions, and were studied in more details by Gokhberg and Kre\u{\i}n \cite{Gokhberg:Krein:1965} (see also \cite{Grothendieck:1955}). First explicit study of those ideals which are direct noncommutative analogues of Orlicz spaces is due to Rao \cite{Rao:1971,Rao:Ren:1991}, where $\Orlicz$ is assumed to be a  continuous modular function, $\Orlicz(\ab{x})$ for $x\in\BH$ is understood in terms of the spectral representation, an analogue of the MNTL norm reads
\begin{equation}
	\BH\ni x\mapsto\n{x}_\Orlicz:=\inf\left\{\lambda>0\mid\tr\left(\Orlicz\left(\frac{\ab{x}}{\lambda}\right)\right)\leq1\right\},
\end{equation}
while the noncommutative Orlicz space is defined as
\begin{equation}
	\schatten_\Orlicz(\H):=\{x\in\BH\mid\n{x}_\Orlicz<\infty\}.
\end{equation}
The generalisation of Orlicz spaces to semi-finite W$^*$-algebras $\N$ equipped with a faithful normal semi-finite trace $\tau$ were proposed by Muratov \cite{Muratov:1978,Muratov:1979}, Dodds, Dodds, and de Pagter \cite{Dodds:Dodds:dePagter:1989}, and Kunze \cite{Kunze:1990}. Two latter constructions are based on the results of Fack and Kosaki \cite{Fack:Kosaki:1986}. Given any $y\in\MMM(\N,\tau)$, the \df{rearrangement function} is defined as \cite{Grothendieck:1955} (see also \cite{Yeadon:1975})
\begin{equation}
	\rearr{y}{\tau}:[0,\infty[\,\ni t\mapsto\rearr{y}{\tau}(t):=\inf\{s\geq0\mid\tau(\pvm^{\ab{x}}(]s,+\infty[)\leq t\}\in[0,\infty].
\end{equation}
If $x\in\MMM(\N,\tau)^+$ and $f:[0,\infty[\ra[0,\infty[$ is a continuous nondecreasing function, then \cite{Fack:Kosaki:1986}
\begin{align}
	\tau(f(x))&=\int_0^\infty\dd t f(\rearr{x}{\tau}(t)),\label{FK.property.one}\\
	\rearr{f(x)}{\tau}(t)&=f(\rearr{x}{\tau}(t))\;\;\forall t\in\RR^+.\label{FK.property.two}
\end{align}
Using this result, Kunze \cite{Kunze:1990} defined a noncommutative Orlicz space associated with an arbitrary Orlicz function $\Orlicz$ as
\begin{equation}
	L_\Orlicz(\N,\tau):=\Span_\CC\{x\in\MMM(\N,\tau)\mid\tau(\Orlicz(\ab{x}))\leq1\},
\label{Kunze.nc.Orlicz}
\end{equation}
equipped with a quantum version of a MNTL norm,
\begin{equation}
	\n{\cdot}_\Orlicz:\MMM(\N,\tau)\ni x\mapsto\inf\{\lambda>0\mid\tau(\Orlicz(\lambda^{-1}\ab{x}))\leq1\},
\end{equation}
under which, as he proves, \eqref{Kunze.nc.Orlicz} is a Banach space. From linearity, it follows that
\begin{equation}
	L_\Orlicz(\N,\tau)=\{x\in\MMM(\N,\tau)\mid\exists\lambda>0\;\;\tau(\Orlicz(\lambda\ab{x}))<\infty\}.
\end{equation}
On the other hand,  Dodds, Dodds, and de Pagter \cite{Dodds:Dodds:dePagter:1989} defined (implicitly) a noncommutative Orlicz space associated with $(\N,\tau)$ and an Orlicz function $\Orlicz$ as
\begin{equation}
	L_\Orlicz(\N,\tau):=\{x\in\MMM(\N,\tau)\mid\rearr{x}{\tau}\in L_\Orlicz(\RR^+,\mho_{\mathrm{Borel}}(\RR^+),\dd\lambda)\}.
\label{DDdP.nc.Orlicz}
\end{equation}
By means of \eqref{FK.property.one} and \eqref{FK.property.two}, these two definitions are equivalent\footnote{See a discussion in \cite{Labuschagne:Majewski:2011,Labuschagne:2013} of the case when continuity of $\Orlicz$ is relaxed to continuity on $[0,x_\Orlicz[$ and left continuity at $x_\Orlicz$ with $x_\Orlicz\neq+\infty$.}. Kunze \cite{Kunze:1990} showed that, for $\Orlicz$ satisfying global $\triangle_2$ condition,
\begin{align}
	L_\Orlicz(\N,\tau)&=\{x\in\MMM(\N,\tau)\mid\tau(\Orlicz(\ab{x}))<\infty\},\\
	L_\Orlicz(\N,\tau)&=E_\Orlicz(\N,\tau):=\overline{\N\cap L_\Orlicz(\N,\tau)}^{\n{\cdot}_\Orlicz},\\
	(L_\Orlicz(\N,\tau))^\banach&\iso L_{\Orlicz^\Young}(\N,\tau).
\end{align}
In \cite{Ayupov:Chilin:Abdullaev:2012} it is shown that if $\tau_1$ and $\tau_2$ are faithful normal semi-finite traces on a semi-finite W$^*$-algebra $\N$, and $\Orlicz$ is an Orlicz function satisfying global $\triangle_2$ condition, then $L_\Orlicz(\N,\tau_1)$ and $L_\Orlicz(\N,\tau_2)$ are isometrically isomorphic. Further analysis of the structure of $L_\Orlicz(\N,\tau)$ spaces in the context of modular function was provided by Sadeghi \cite{Sadeghi:2012}, who showed that the map $\MMM(\N,\tau)\ni x\mapsto \tau(\Orlicz(\ab{x}))\in[0,\infty]$ is a modular function for any Orlicz function $\Orlicz$. He also notes that the results of \cite{Chilin:Krygin:Sukochev:1992} and \cite{Dodds:Dodds:dePagter:1993} allow to infer, respectively, the uniform convexity and reflexivity of the spaces $(L_\Orlicz(\N,\tau),\n{\cdot}_\Orlicz)$ from the corresponding properties of the commutative Orlicz spaces $(L_\Orlicz(\RR^+,\mho_{\mathrm{Borel}}(\RR^+),\dd\lambda),\n{\cdot}_\Orlicz)$. This leads to\footnote{The statement of the sufficient condition for reflexivity in Collorary 4.3 of \cite{Sadeghi:2012} is missing the requirement of the global $\triangle_2$ condition for $\Orlicz^\Young$.}
\begin{corollary}
 $(L_\Orlicz(\N,\tau),\n{\cdot}_\Orlicz)$ is: uniformly convex if $\Orlicz$ is uniformly convex and satisfies global $\triangle_2$ condition; reflexive if $\Orlicz$ and $\Orlicz^\Young$ satisfy global $\triangle_2$ condition.
\end{corollary}

Al-Rashed and Zegarli\'{n}ski \cite{AlRashed:Zegarlinski:2007,AlRashed:Zegarlinski:2011} proposed a construction of a family of noncommutative Orlicz spaces associated with a faithful normal state on a countably finite W$^*$-algebra. Ayupov, Chilin and Abdullaev \cite{Ayupov:Chilin:Abdullaev:2012} proposed the construction of a family of noncommutative Orlicz spaces $L_\Orlicz(\N,\psi)$ for a semi-finite W$^*$-algebra $\N$, a faithful normal locally finite weight $\psi$, and an Orlicz function $\Orlicz$ satisfying global $\triangle_2$ condition. This construction extends Trunov's theory of $L_p(\N,\tau)$ spaces \cite{Trunov:1979,Trunov:1981,Zolotarev:1988}. Labuschagne \cite{Labuschagne:2013} provided a construction of the family of noncommutative Orlicz spaces $L_\Orlicz(\N,\psi)$ associated with an arbitrary W$^*$-algebra and a faithful normal semi-finite weight $\psi$.  The problem of the canonical (weight-independent and functorial) construction of  noncommutative Orlicz spaces over arbitrary W$^*$-algebras remains still open.

\section{Orlicz spaces over standard core\label{functoriality.section}}

\begin{definition}
For an arbitrary W$^*$-algebra $\N$ and arbitrary Orlicz function $\Orlicz$, let a \df{noncommutative core Orlicz space} be defined as a vector space
\begin{equation}
	L_\Orlicz(\core):=\{x\in\MMM(\core,\taucore)\mid\exists\lambda>0\;\;\taucore(\Orlicz(\lambda\ab{x}))<\infty\},\label{nc.Orlicz.space}
\end{equation}
equipped with the norm
\begin{equation}
	\n{\cdot}_\Orlicz:\MMM(\core,\taucore)\ni x\mapsto\inf\{\lambda>0\mid\taucore(\Orlicz(\lambda^{-1}\ab{x}))\leq1\}.\label{nc.Orlicz.norm}
\end{equation}
In addition, we define
\begin{equation}
	E_\Orlicz(\core):=\overline{\N\cap L_\Orlicz(\core)}^{\n{\cdot}_\Orlicz}.
\end{equation}
\end{definition}
Because $\core$ is a semi-finite von Neumann algebra, while $\taucore$ is a faithful normal semi-finite trace on $\core$, all above results on the Banach space structure of $L_\Orlicz(\N,\tau)$ immediately apply to $L_\Orlicz(\core)$. In particular, if $\Orlicz$ satisfies global $\triangle_2$ condition, then $L_\Orlicz(\core)=E_\Orlicz(\core)$ and $L_\Orlicz(\core)^\banach\iso L_{\Orlicz^\Young}(\core)$. Moreover: if $\Orlicz$ is also uniformly convex, then $L_\Orlicz(\core)$ is uniformly convex and $L_{\Orlicz^\Young}(\core)$ is uniformly Fr\'{e}chet differentiable; if $\Orlicz^\Young$ also satisfies global $\triangle_2$ condition, then $L_\Orlicz(\core)$ is reflexive. 
For any choice of a normal semi-finite weight $\psi$ on $\N$, $\core$ can be represented as $\N\rtimes_{\sigma^\psi}\RR$ by means of \eqref{core.iso.map}. Hence, for any choice of $\psi$, $L_\Orlicz(\core)$ can be represented as isometrically isomorphic to Kunze's noncommutative Orlicz spaces $L_\Orlicz(\N\rtimes_{\sigma^\psi}\RR)$. 


In what follows, will first recall Kosaki's construction of canonical representation of W$^*$-algebra. Then we will use it to extend functoriality of the Falcone--Takesaki noncommutative flow of weights. (While Kosaki's construction of canonical $L_2(\N)$ spaces uses normal states instead of elements of $\W_0(\N)$ applied in the Falcone--Takesaki $L_2(\N)$, the algebraic character of its structure allows for using it for the purpose of canonical representation of any W$^*$-algebra as a von Neumann algebra, which is the starting point required for the Falcone--Takesaki construction.) Finally, we will show that $*$-isomorphisms of W$^*$-algebras are canonically mapped to isometric isomorphisms of $L_\Orlicz(\core)$ spaces. 

Following Kosaki \cite{Kosaki:1980:PhD}, consider new addition and multiplication structure on $\N_\star^+$,
\begin{align}
        \lambda\sqrt{\phi}&=
        \sqrt{\lambda^2\phi}
        \;\;\forall\lambda\in\RR^+
        \;\forall\phi\in\N_\star^+,
        \label{Kosaki.new.multiplication}\\
        \sqrt{\phi}+\sqrt{\psi}&=
        \sqrt{(\phi+\psi)(y^*\,\cdot\,y)}
        \;\;\forall\phi,\psi\in\N_\star^+,
        \label{Kosaki.new.addition}
\end{align}
where
\begin{equation}
        y:=\Connes{\phi}{(\phi+\psi)}{-\ii/2}+\Connes{\psi}{(\phi+\psi)}{-\ii/2},
\end{equation}
and $\sqrt{\phi}$ is understood as a \textit{symbol} denoting the element $\phi$ of $\N_\star^+$ whenever it is subjected to the above operations instead of `ordinary' addition and multiplication on $\N_\star^+$. A `noncommutative Hellinger integral' on $\N_\star^+$,
\begin{equation}
        \kosaki{\phi}{\psi}:=(\phi+\psi)\left(\Connes{\psi}{(\phi+\psi)}{-\ii/2}^*\Connes{\phi}{(\phi+\psi)}{-\ii/2}\right)
\end{equation}
is a positive bilinear symmetric form on $\N_\star^+$ with respect to the operations defined by \eqref{Kosaki.new.multiplication} and \eqref{Kosaki.new.addition}. Consider an equivalence relation $\sim_\surd$ on pairs $(\sqrt{\phi},\sqrt{\psi})\in\N_\star^+\times\N_\star^+$,
\begin{equation}
        (\sqrt{\phi_1},\sqrt{\psi_1})\sim_\surd(\sqrt{\phi_2},\sqrt{\psi_2})\;\;\iff\;\;\sqrt{\phi_1}+\sqrt{\psi_2}=\sqrt{\phi_2}+\sqrt{\psi_1}.
\end{equation}
The set of equivalence classes $\N_\star^+\times\N_\star^+/\sim_\surd$ can be equipped with a real vector space structure, provided by
\begin{align}
(\sqrt{\phi_1},\sqrt{\phi_2})_\surd+(\sqrt{\psi_1},\sqrt{\psi_2})_\surd&:=(\sqrt{\phi_1}+\sqrt{\psi_1},\sqrt{\phi_2}+\sqrt{\psi_2})_\surd,\\
\lambda\cdot(\sqrt{\phi},\sqrt{\psi})_\surd&:=
        \left\{
                \begin{array}{ll}
                        (\lambda\sqrt{\phi},\lambda\sqrt{\psi})_\surd&:\lambda\geq0\\
                        ((-\lambda)\sqrt{\psi},(-\lambda)\sqrt{\phi})_\surd&:\lambda<0,
                \end{array}
        \right.
\end{align}
where $(\sqrt{\phi},\sqrt{\psi})_\surd$ denotes an element of $\N_\star^+\times\N_\star^+/\sim_\surd$. The real vector space $(\N_\star^+\times\N_\star^+/\sim_\surd,+,\cdot)$ will be denoted $V$. The map
\begin{equation}
        \N_\star^+\ni\phi\mapsto(\sqrt{\phi},0)_\surd\in V
\label{positive.predual.canonical.embedding}
\end{equation}
is injective, positive and preserves addition and multiplication by positive scalars. Its image in $V$ will be denoted by $L_2(\N)^+$. A function
\begin{equation}
        \s{\cdot,\cdot}_\surd:V\times V\ra\RR,
\end{equation}
\begin{align}
        \s{(\sqrt{\phi_1},\sqrt{\phi_2})_\surd,(\sqrt{\psi_1},\sqrt{\psi_2})_\surd}_\surd:&=
        \kosaki{\phi_1}{\psi_1}+
        \kosaki{\phi_1}{\psi_2}\nonumber\\&+
        \kosaki{\phi_2}{\psi_1}+
        \kosaki{\phi_2}{\psi_2},
\end{align}
is an inner product on $V$, and $(V,\s{\cdot,\cdot}_\surd)$ is a real Hilbert space with respect to it, denoted $L_2(\N;\RR)$. The \df{canonical Hilbert space} is defined as a complexification of the Hilbert space $L_2(\N;\RR)$,
\begin{equation}
        L_2(\N):=L_2(\N;\RR)\otimes_\RR\CC.
\end{equation}
The space $L_2(\N)^+$ is a self-polar convex cone in $L_2(\N)$, and, by \eqref{positive.predual.canonical.embedding}, it is an embedding of $\N_\star^+$ into $L_2(\N)$. The elements of $L_2(\N)^+$ will be denoted $\phi^{1/2}$, where $\phi\in\N^+_\star$. Every element of $L_2(\N)$ can be expressed as a linear combination of four elements of $L_2(\N)^+$. The antilinear conjugation $J_\N:L_2(\N)\ra L_2(\N)$ is defined by
\begin{equation}
        J_\N(\xi+\ii\zeta)=\xi-\ii\zeta\;\;\forall\xi,\zeta\in L_2(\N;\RR).
\end{equation}
The quadruple $(L_2(\N),\N,J_\N,L_2(\N)^+)$ is a standard form of $\N$, called a \df{canonical standard form} of $\N$. 

A bounded generator $\der_\N(x):=\frac{\dd}{\dd t}\left(f(\ee^{tx})\right)|_{t=0}$ of the norm continuous one parameter group of automorphisms
\begin{equation}
        \RR\ni t\mapsto f(\ee^{tx})\in\BBB(L_2(\N))\;\;\forall x\in\N,
\end{equation}
where $\N\ni x\mapsto f(x)\in\BBB(L_2(\N))$ is defined as a unique extension of the bounded linear function
\begin{equation}
        L_2(\N)^+\ni\phi^{1/2}\mapsto
        \left(\phi\left(
        \sigma^\phi_{+\ii/2}(x)x^*\,\cdot\,x\sigma^\phi_{-\ii/2}(x^*)
        \right)\right)^{1/2}\in L_2(\N)^+\;\;
        \forall x\in\N\;\forall\phi\in\N_\star^+,
\end{equation}
determines a map
\begin{equation}
        \der_\N:\N\ni x\mapsto\der_\N(x)\in\BBB(L_2(\N)),
\end{equation}
which is a homomorphism of real Lie algebras: for all $x,y\in\N$ and for all $\lambda\in\RR$,
\begin{align}
        [\der_\N(x),\der_\N(y)]&=\der_\N([x,y]),\\
        \lambda\der_\N(x)&=\der_\N(\lambda x).
\end{align}
The faithful normal representation $\pi_\N:\N\ra\BBB(L_2(\N))$,
\begin{equation}
        \pi_\N(x):=\textstyle\frac{1}{2}(\der_\N(x)-\ii\der_\N(\ii x)),
\end{equation}
determines a standard representation $(L_2(\N),\pi_\N,J_\N,L_2(\N)^+)$ of a W$^*$-algebra $\N$, called a \df{canonical representation} of $\N$. Every $*$-isomorphism $\varsigma:\N_1\ra\N_2$ of W$^*$-algebras $\N_1,\N_2$ determines a unique unitary equivalence $u_\varsigma:L_2(\N_2)\ra L_2(\N_1)$ satisfying $u_\varsigma(L_2(\N_2)^+)=L_2(\N_1)^+$, $u_\varsigma^*J_{\N_2}u_\varsigma=J_{\N_1}$, and such that $\Ad(u_\varsigma^*)$ is a unitary implementation of $\varsigma$. This means that Kosaki's construction of canonical representation defines a functor $\CanRep$ from the category $\WsIso$ of W$^*$-algebras with $*$-isomorphisms to the category $\StdRep$ of standard representations with standard unitary equivalences. It also allows to define a functor $\CanVN$ from the category $\Wsn$ of W$^*$-algebras with normal $*$-homomorphisms to the category $\VNn$ of von Neumann algebras with normal $*$-homomorphisms. The functor $\CanVN$ assigns $\pi_\N(\N)$ to each $\N$, and normal $*$-homomorphism $\pi_\N\circ\varsigma\circ\pi_\N^{-1}:\N_1\ra\N_2$ to each normal $*$-homomorphism $\varsigma:\N_1\ra\N_2$ (which is well defined due to faithfulness of $\pi_\N$). Let $\FrgHlb:\VNn\ra\Wsn$ be the forgetful functor which forgets about Hilbert space structure that underlies von Neumann algebras and their normal $*$-homomorphisms. Due to Sakai's theorem \cite{Sakai:1956}, $\CanVN$ and $\FrgHlb$ form the equivalence of categories,
\begin{equation}
        \FrgHlb\circ\CanVN\iso\id_{\Wsn},\;\;\CanVN\circ\FrgHlb\iso\id_{\VNn}.
\label{Sakai.Kosaki.duality}
\end{equation}

Consider the category $\WstarCovRTr$ of quadruples $(\N,\RR,\alpha,\tau)$, where $\N$ is a semi-finite W$^*$-algebra, $\tau$ is a faithful normal semi-finite trace on $\N$, and $(\N,\RR,\alpha)$ is a W$^*$-dynamical system, with morphisms
\begin{equation}
        (\N_1,\RR,\alpha^1,\tau_1)\ra
        (\N_2,\RR,\alpha^2,\tau_2)
\end{equation}
given by such $*$-isomorphisms $\varsigma:\N_1\ra\N_2$ which satisfy
\begin{align}
        \varsigma\circ\alpha^1_t&=\alpha^2_t\circ\varsigma\;\;\forall t\in\RR,\label{iso.covariance}\\
        \tau_1&=\tau_2\circ\varsigma\label{trace.covariance}.
\end{align}
Falcone and Takesaki call the quadruple $(\core,\RR,\tilde{\sigma},\taucore)$ a \df{noncommutative flow of weights}, and prove that every $*$-isomorphism $\varsigma:\N_1\ra\N_2$ of von Neumann algebras extends to a $*$-isomorphism $\widetilde{\varsigma}:\core_1\ra\core_2$ satisfying \eqref{iso.covariance} and \eqref{trace.covariance}. This defines a functor 
\begin{equation}
        \FTflow:\VNIso\ra\WstarCovRTr,
\end{equation}
where $\VNIso$ is a category of von Neumann algebras with $*$-isomorphisms. The restriction of $\tilde{\sigma}$ to the center $\zentr_\core$ is the Connes--Takesaki flows of weights $(\zentr_\core,\RR,\tilde{\sigma}|_{\zentr_\core})$ \cite{Falcone:Takesaki:2001}. Hence, the relationship between the Falcone--Takesaki noncommutative flow of weights and the Connes--Takesaki flow of weights can be summarised in terms of the commutative diagram
\begin{equation}
\xymatrix{
        \WsIso
        \ar[rr]^{\CanVN}
        &&
        \VNIso
        \ar[rr]^{\FTflow}
        &&
        \WstarCovRTr
        \ar[d]^{\zentr\circ\ForgTr}
        \\
        \WsfIIIIso
        \ar@{ >->}[u]
        \ar[rr]_{\CanVN}
        &&
        \VNfIIIIso
        \ar@{ >->}[u]
        \ar[rr]_{\CTflow}
        &&
        \WstarCovR,     
}
\label{ctft.cat.diag}
\end{equation}
where $\WsfIIIIso$ (respectively, $\VNfIIIIso$) is a category of type III factor W$^*$-algebras (respectively, von Neumann algebras) with $*$-isomorphisms, $\ForgTr$ denotes the forgetful functor that forgets about traces, while $\zentr:\WstarCovR\ra\WstarCovR$ is an endofunctor that assigns an object $(\zentr_\N,\RR,\alpha|_{\zentr_\N})$ to each $(\N,\RR,\alpha)$, and assigns a morphism $\varsigma^{12}_\zentr$ such that
\begin{equation}
        \varsigma^{12}_\zentr\circ\alpha^1_t|_{\zentr_{\N_1}}=
        \alpha^2_t|_{\zentr_{\N_2}}\circ\varsigma^{12}_\zentr
\end{equation}
to each $\varsigma:(\N_1,\RR,\alpha^1)\ra(\N_2,\RR,\alpha^2)$. 

\begin{proposition}
Every $*$-isomorphism $\varsigma:\N_1\ra\N_2$ of W$^*$-algebras gives rise to a corresponding isometric isomorphism $L_\Orlicz(\core_1)\ra L_\Orlicz(\core_2)$.
\end{proposition}
\begin{proof}
By the Falcone--Takesaki construction, and its composition \eqref{ctft.cat.diag} with Kosaki's construction, $\varsigma:\N_1\ra\N_2$ induces a $*$-isomorphism $\widetilde{\varsigma}:\core_1\ra\core_2$ of semi-finite von Neumann algebras and a mapping $(\core_1,\RR,\tilde{\sigma}^1,\taucore_1)\ra(\core_2,\RR,\tilde{\sigma}^2,\taucore_2)$ satisfying
\begin{align}
	\widetilde{\varsigma}\circ\tilde{\sigma}^1_t&=\tilde{\sigma}^2_t\circ\widetilde{\varsigma}\;\;\forall t\in\RR,\label{widetilde.covariance.insanity.one}\\
	\taucore_1&=\taucore_2\circ\widetilde{\varsigma}.\label{widetilde.covariance.insanity.two}
\end{align}
By Collorary 38 in \cite{Terp:1981}, every $*$-isomorphism of semi-finite von Neumann algebras satisfying \eqref{widetilde.covariance.insanity.one} and \eqref{widetilde.covariance.insanity.two} extends to a topological $*$-isomorphism of corresponding spaces of $\tau$-measurable operators affiliated with these algebras, and this extension preserves the property \eqref{widetilde.covariance.insanity.two}. The $*$-isomorphism $\widetilde{\varsigma}$ extends to $\bar{\varsigma}:\MMM(\core_1,\taucore_1)\ra\MMM(\core_2,\taucore_2)$ by $\bar{\varsigma}(\cdot)=u(\cdot)u^*$, where $u$ is a unitary operator implementing $\widetilde{\varsigma}(\cdot)=u(\cdot)u^*$. It remains to show that $\bar{\varsigma}$ is an isometric isomorphism. Using \eqref{FK.property.one}, we can rewrite \eqref{nc.Orlicz.norm} as
\begin{equation}
	\n{x}_\Orlicz=\inf\{\lambda>0\mid\int_0^\infty\dd t\Orlicz(\lambda^{-1}\rearr{\ab{x}}{\taucore}(t))<\infty\},
\end{equation}
where
\begin{equation}
	\rearr{\ab{x}}{\taucore}(t)=\inf\{s\geq0\mid\taucore(\pvm^{\ab{x}}(]s,+\infty[))\leq t\}.
\end{equation}
For $\MMM(\core_2,\taucore_2)=\MMM(\widetilde{\varsigma}(\core_1),\taucore_1\circ\widetilde{\varsigma}^{-1})$ and $x\in\MMM(\core_1,\taucore_1)$ we have
\begin{equation}
\taucore_1\circ\widetilde{\varsigma}^{\,-1}\left(\pvm^{\ab{\widetilde{\varsigma}(x)}}(s,+\infty[)\right)=\taucore_1\circ\widetilde{\varsigma}^{\,-1}\circ\widetilde{\varsigma}(\pvm^{\ab{x}}(]s,+\infty[))=\taucore_1(\pvm^{\ab{x}}(]s,+\infty[)).
\end{equation}
Hence, $\bar{\varsigma}:L_\Orlicz(\core_1)\ra L_\Orlicz(\core_2)$ is an isometric isomorphism.
\end{proof}

\begin{corollary}
Denoting the category of noncommutative Orlicz spaces $L_\Orlicz(\core)$ with isometric isomorphisms by $\catname{nc\widetilde{L}}_\Orlicz\catname{Iso}$, we conclude that our construction determines a functor
\begin{equation}
	\mathrm{nc\widetilde{L}}_\Orlicz:\WsIso\ra\catname{nc\widetilde{L}}_\Orlicz\catname{Iso}.
\end{equation}
\end{corollary}

Following the results of Sherman \cite{Sherman:2005} (generalising earlier results of \cite{Yeadon:1981} and \cite{Watanabe:1992}), we end this section with an interesting problem: for which Orlicz functions $\Orlicz$ there exists a functor from $\catname{nc\widetilde{L}}_\Orlicz\catname{Iso}$ to the category of W$^*$-algebras with surjective Jordan $*$-isomorphisms? And, furthermore, for which $\Orlicz$ these functorial relationships can be extended to category $\catname{ncL}_\Orlicz\catname{Iso}$ of all noncommutative Orlicz spaces with isometric isomorphisms?
\section*{Acknowledgments}
I am very indebted to Professor Stanis{\l}aw L. Woronowicz for introducing me to Tomita--Takesaki and Falcone--Takesaki theories. I would like also to thank W{\l}adys{\l}aw Adam Majewski, Mustafa Muratov, Dmitri\u{\i} Pavlov, and David Sherman for discussion and correspondence, as well as two anonymous referees of the previous version of this paper for their comments. This research was supported in part by Perimeter Institute for Theoretical Physics. Research at Perimeter Institute is supported by the Government of Canada through Industry Canada and by the Province of Ontario through the Ministry of Research and Innovation.

\section*{References}
\addcontentsline{toc}{section}{References}
{%
\scriptsize


All Cyrillic titles and names were transliterated from the original papers and books. For the Latin transliteration of the Cyrillic script (in references and surnames) we use the following modification of the system GOST 7.79-2000B: {\cyrrm{ts}} = c, {\cyrrm{ch}} = ch, {\cyrrm{kh}} = kh, {\cyrrm{zh}} = zh, {\cyrrm{sh}} = sh, {\cyrrm{shch}} = shh, {\cyrrm{yu}} = yu, {\cyrrm{ya}} = ya, {\cyrrm{\"{e}}} = \"{e}, {\cyrrm{\cdprime}} = `, {\cyrrm{\cprime}} = ', {\cyrrm{\`{e}}} = \`{e}, {\cyrrm{\u{i}}} = \u{\i}, with an exception that names beginning with {\cyrrm{Kh}} are transliterated to H. For Russian texts: {\cyrrm{y}} = y, {\cyrrm{i}} = i; for Ukrainian: {\cyrrm{i}} = y, i = i, \"{\i} = \"{\i}.

\begingroup
\raggedright
\renewcommand\refname{\vskip -1cm}

\endgroup        
}%
\end{document}